\documentclass[12 pt]{elsarticle}
\usepackage{fullpage,graphicx, amsmath,amsthm,caption,epstopdf,tikz,enumitem,color,subcaption, lineno, float,comment, amsfonts, wasysym}
\usepackage[mathscr]{euscript}
\setlist{nolistsep}
\usepackage{hyperref}
\usepackage[capitalise, noabbrev, nameinlink]{cleveref}

\usepackage{rotating}
\usepackage{thm-restate}
\definecolor{purple}{RGB}{138,43,226}
\usetikzlibrary{decorations.pathmorphing}
\usetikzlibrary{shapes,backgrounds,arrows.meta,decorations,graphs,graphs.standard,quotes,backgrounds}
\usetikzlibrary{decorations.markings}
\tikzset{snake it/.style={decorate, decoration=snake}}

\newlist{statement}{enumerate}{1}
\setlist[statement]{label=(\roman*),ref=\doublelabel{(\roman*)},before=\setcrefdoublealias, nolistsep}
\crefname{statementi}{statement}{statements} 

\makeatletter
\newcommand*\doublelabel[1]{\protect\@twolabels{#1}{\@currentlabel~#1}}
\let\@twolabels\@firstoftwo

\def\setcrefdoublealias{%
	\begingroup\edef\x{\endgroup%
		\noexpand\crefalias{statementi}{%
			\noexpand\protect\noexpand\@twolabels%
			{statementi}{\expandafter\@extractcounterfromcreflabel\cref@currentlabel\end@extractcounterfromcreflabel}%
		}%
	}\x%
}
\def\@extractcounterfromcreflabel[#1]#2\end@extractcounterfromcreflabel{#1}

\newcommand*\versionWithTheorem[1]{\@ifstar{\versionWithTheorem@aux{#1*}}{\versionWithTheorem@aux{#1}}}
\newcommand*\versionWithTheorem@aux[2]{\begingroup\let\@twolabels\@secondoftwo#1{#2}\endgroup}
\makeatother

\DeclareRobustCommand*\crefWithTheorem{\versionWithTheorem\cref}

\DeclareFontFamily{U}{FdSymbolA}{}
\DeclareFontShape{U}{FdSymbolA}{m}{n}{
	<-> s * [1.5] FdSymbolA-Book
}{}
\DeclareFontShape{U}{FdSymbolA}{m}{b}{
	<-> s * [1] FdSymbolA-Medium
}{}
\DeclareSymbolFont{fdsymbols}{U}{FdSymbolA}{m}{n}
\SetSymbolFont{fdsymbols}{bold}{U}{FdSymbolA}{m}{b}
\DeclareMathSymbol{\upY}{\mathbin}{fdsymbols}{45}
\DeclareMathSymbol{\downY}{\mathbin}{fdsymbols}{47}
\DeclareMathSymbol{\hourglass}{\mathbin}{fdsymbols}{43}

\mathchardef\mhyphen="2D

\definecolor{Green}{RGB}{34, 139, 34}
\definecolor{DarkGray}{RGB}{136,136,136}
\definecolor{LightGray}{RGB}{192,192,192}

\newtheorem{theorem}{Theorem}[section]
\newtheorem{lemma}[theorem]{Lemma}

\newtheorem{definition}[theorem]{Definition}
\newtheorem{claim}{Claim}
\numberwithin{figure}{section}

\newtheorem{observation}[theorem]{Observation}
\crefname{observation}{Observation}{Observations}
\crefname{claim}{Claim}{Claims}
\newtheoremstyle{TheoremNum}
{\topsep}{\topsep}              
{\itshape}                      
{}                              
{\bfseries}                     
{.}                             
{ }                             
{\thmname{#1}\thmnote{ \bfseries #3}}
\theoremstyle{TheoremNum}

\newcommand{\TFinf}{\mathscr{F}_{\infty}}

\newcommand{\Loneinf}{\mathscr{L}_{\infty}}
\newcommand{\Ltwoinf}{\mathscr{L}^{\Delta}_{\infty}}
\newcommand{\Lthreeinf}{\mathscr{L}^{\nabla\Delta}_{\infty}}

\newcommand{\N}{\mathbb{N}}

\newcommand{\eq}{=}

\definecolor{Green}{RGB}{34, 139, 34}

\journal{Discrete Mathematics}
\begin{document}
\begin{frontmatter}
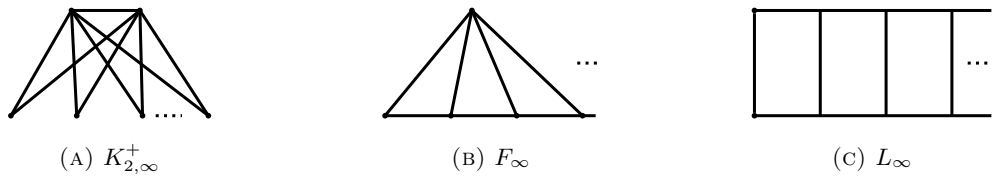
\title{Unavoidable Induced Subgraphs of Infinite 2-Connected Graphs}

\author[label1]{Sarah Allred\corref{cor1}%
}
\ead{sarahallred@southalabama.edu}
\affiliation[label1]{organization={Department of Mathematics and Statistics,  University of South Alabama},
            addressline={}, 
            city={Mobile},
            postcode={36688}, 
            state={AL},
            country={USA}}

\author[label2]{Guoli Ding}
\ead{ding@math.lsu.edu}
\affiliation[label2]{organization={Department of Mathematics, Louisiana State University},
            addressline={}, 
            city={Baton Rouge},
            postcode={70803}, 
            state={LA},
            country={USA}}

\author[label2]{Bogdan Oporowski}
\ead{bogdan@math.lsu.edu}

\cortext[cor1]{Corresponding author}

\begin{abstract} In 1930, Ramsey proved that every infinite graph contains either an infinite clique or an infinite independent set. K\"{o}nig proved that every connected infinite graph contains either a ray or a vertex of infinite degree.  In this paper, we establish the 2-connected analog of these results. \end{abstract}

\end{frontmatter}

\section{Introduction}\label{sec:intro}
The terms and symbols that are not defined explicitly in this paper will be understood as defined in \cite{west}.

Let us begin with the classical Ramsey Theorem, \cite{ramsey}.

\begin{theorem}
\label{thm:ramsey}
For every positive integer $r$, there is an integer $f_{\ref{thm:ramsey}}(r)$
such that every graph on at least $f_{\ref{thm:ramsey}}(r)$ vertices contains as an induced subgraph either $K_r$ or its complement $\overline K_r$.
\end{theorem}

One analog of this theorem is the following result for  connected graphs, instead of arbitrary graphs.

\begin{theorem}
\label{thm:doov}
For every positive integer $r$, there is an integer $f_{\ref{thm:doov}}(r)$
such that every connected graph on at least $f_{\ref{thm:doov}}(r)$ vertices
contains an induced $K_r$, $K_{1,r}$, or $P_r$.
\end{theorem}

For higher connectivities, similar results are known, but they use different containment relations on graphs. In particular, for $k=2,3,4$, unavoidable large $k$-connected minors and for $j=2,3$, $j$-connected topological-minors are determined in \cite{Unavoidabletopminor3conngraphs}; the corresponding parallel-minors are determined in \cite{unavoidableparminor4conngraphs}. These results are very useful in analyzing graph structures concerning minors. However, they provide very little help in dealing with induced subgraphs. For instance, if we are interested in the behavior of a large 2-connected claw-free graph, then knowing that such a graph has a large $C_n$ or $K_{2,n}$ minor is not of much help. In such a situation, it is desirable to have the unavoidable large 2-connected induced subgraphs. In an earlier paper \cite{unavoidableinducedsubgraphs}, the authors of this paper determined all these unavoidable graphs, which we describe below.

We define two families of graphs that generalize, respectively, stars and paths in  \cref{thm:doov}. Let $n$ be an integer exceeding two. First, let $\Theta_n$ be the family of graphs $G$ such that $G$ consists of two specified vertices and $n$ pairwise internally disjoint paths between those two vertices. These graphs can be considered as a generalization of stars. Next, let $\Lambda_n$ be the family of graphs $G$ constructed as follows. Let $G_1,\dots ,G_k$ where $k\ge n$ be disjoint graphs, such that each $G_i$ is isomorphic to $K_3$ or $K_4$. For each $i=1,\dots,k$, let $e^i$ and $f^i$ be distinct edges of $G_i$, where $e^i$ is not adjacent to $f^i$ if $G_i\cong K_4$. Then $G$ is obtained by identifying $f^i$ with $e^{i+1}$, for all $i=1,\dots, k-1$, and then deleting some (could be all or none) of the identified edges. Note that each cycle of length at least $n+2$ is a member of $\Lambda_n$. These graphs can be considered as a generalization of paths.

With these definitions in mind, the main result of~\cite{unavoidableinducedsubgraphs}
may be restated as follows.

\begin{restatable}{theorem}{thmfinite}\label{thm:finite}
For every integer $r$ exceeding two, there is an integer $f_{\ref{thm:finite}}(r)$ such that every $2$-connected graph of order at least $f_{\ref{thm:finite}}(r)$ contains $K_r$ or a member of $\Theta_r\cup \Lambda_r$ as an induced subgraph.
\end{restatable}

It is worth pointing out that there are a few similar results for induced subgraphs. 
In \cite{homogeneous}, unavoidable large homogeneous-free induced subgraphs are determined, and in \cite{doublecon}, unavoidable large doubly-connected induced subgraphs are determined.  \cref{thm:finite} is an addition to this line of results.

The goal of this paper is to extend \cref{thm:finite} to infinite graphs. 
For infinite graphs, the terms \emph{subgraph}, \emph{induced subgraph}, and \emph{subdivision} are defined in the natural way found in \cite{west}. 
It should be emphasized that when an edge is subdivided,  it may be subdivided only finitely many times. 
This paper will consider only simple graphs. 
That is, we implicitly assume that every involved graph has no multiple edges and no loops. 

Note that \cref{thm:ramsey} has a natural analog for infinite graphs. 
The theorems for infinite graphs in this paper apply to graphs with arbitrary infinite cardinality. 
However, the graphs presented as unavoidable induced subgraphs, using the subscript $\infty$, are countably infinite.

\begin{theorem}[\cite{ramsey}]
\label{thm:inframsey}
Every infinite graph contains as an induced subgraph either an infinite complete graph $K_\infty$ or its complement $\overline {K}_\infty$.
\end{theorem} 

Let $G$ be an infinite graph. 
We call $G$ \emph{connected} if it has a path between every two of its vertices; we call $G$ \emph{$2$-connected} if $G- v$ is connected for every vertex $v$ of $G$.
A \emph{ray}, denoted $P_\infty$, consists of an infinite sequence of vertices $\{v_1,v_2,v_3,\dots\}$ with edges of the form $v_iv_{i+1}$ for $i\in\mathbb{N}$.
The vertex $v_1$ is the \emph{initial vertex} of the ray.
A graph with no vertices of infinite degree is called \emph{locally-finite}.  
In \cite{konig}, K\"{o}nig proved his Infinity Lemma stated below. 
\begin{theorem}\label{thm:konig}
	Every infinite connected graph has a vertex of infinite degree or contains an induced ray.
\end{theorem}
The following result for infinite connected graphs is an immediate consequence of \cref{thm:inframsey,thm:konig}.

\begin{theorem}\label{thm:infconnectedinduced}
Every infinite connected graph contains an induced $K_\infty$, $K_{1,\infty}$, or $P_\infty$.
\end{theorem}

Let $\Theta_\infty$ be the family of graphs $G$ such that $G$ consists of two specified vertices and an infinite number of pairwise
internally disjoint paths between those two vertices.
Two examples from this family are shown below in \cref{fig:K2infandK2infplus}.
In all figures in this paper, thin line segments represent single edges, and thick line segments represent paths (which may consist of just a single edge).
\begin{figure}[h!]
	\begin{subfigure}[b]{.47\textwidth}
		\begin{center}
			\begin{tikzpicture}
			[scale=.8,auto=left,every node/.style={circle, fill, inner sep=0 pt, minimum size=.75mm, outer sep=0pt},line width=.8mm]
			\node(1) at (2.15,3){};
			\node (2) at (3.45,3){};
			\node (3) at (1,1){};
			\node (4) at (2.25,1){};
			\node(5) at (3.5,1){};
			\node (6) at (4.75,1){};
			\draw[dotted, line width=.4mm] (3.75,1)--(4.25,1);
			\draw (2.15,3)--(1,1); \draw (2.15,3)--(2.25,1);\draw(2.15,3)--(3.5,1);\draw (2.15,3)--(4.75,1);
			\draw (3.45,3)--(1,1); \draw (3.45,3)--(2.25,1);\draw(3.45,3)--(3.5,1);\draw (3.45,3)--(4.75,1);
		\end{tikzpicture}
	\end{center}
	\caption{}
	\label{fig:K2inf}
	\end{subfigure}
\begin{subfigure}[b]{.47\textwidth}
\begin{center}
\begin{tikzpicture}
	[scale=.8,auto=left,every node/.style={circle, fill, inner sep=0 pt, minimum size=.75mm, outer sep=0pt},line width=.8mm]
	\node(1) at (2.15,3){};
	\node (2) at (3.45,3){};
	\node (3) at (1,1){};
	\node (4) at (2.25,1){};
	\node(5) at (3.5,1){};
	\node (6) at (4.75,1){};
	\draw[dotted, line width=.4mm] (3.75,1)--(4.25,1);
	\draw (2.15,3)--(1,1); \draw (2.15,3)--(2.25,1);\draw(2.15,3)--(3.5,1);\draw (2.15,3)--(4.75,1);
	\draw (3.45,3)--(1,1); \draw (3.45,3)--(2.25,1);\draw(3.45,3)--(3.5,1);\draw (3.45,3)--(4.75,1);
	\draw[line width=.4mm] (2.15,3)--(3.45,3);
\end{tikzpicture}
\end{center}
\caption{}
\label{fig:K2infplus}
\end{subfigure}
\caption{Two type of graphs in $\Theta_{\infty}$}
\label{fig:K2infandK2infplus}
\end{figure}

Let $F_\infty$ be the graph obtained from a ray $r_1r_2...$ by adding a new vertex $r$ and joining $r$ with every $r_i$; let $\mathscr F_\infty$ be the family of graphs obtained from $F_\infty$ by subdividing each of the edges an arbitrary number, possibly zero, of times; see \cref{fig:famFinf}. 
Let $F_{\infty}^{\Delta}$ be obtained from $F_\infty$ by first subdividing each $rr_i$ and  $r_ir_{i+1}$ by $s_i$ and $t_i$, respectively, and then joining $s_i$ with $t_i$ for all $i$; let $\mathscr F_\infty^\Delta$ be the family of graphs obtained from $F_\infty^\Delta$ by subdividing each edge that is not in any triangle, an arbitrary number, possibly zero, of times; see \cref{fig:Fdeltinf}. 

\begin{figure}[h!]
	\begin{subfigure}[b]{.47\textwidth}
		\begin{center}
			\begin{tikzpicture}
				[scale=1,auto=left,every node/.style={circle, fill, inner sep=0 pt, minimum size=.75mm, outer sep=0pt},line width=.8mm]
				\node(1) at (2.65,3){};
				\node (3) at (1,1){};
				\node (4) at (2.25,1){};
				\node(5) at (3.5,1){};
				\node (6) at (4.75,1){};
				\draw[dotted, line width=.4mm] (4.95,2)--(5.35,2);
				\draw (2.65,3)--(1,1); \draw (2.65,3)--(2.25,1);\draw(2.65,3)--(3.5,1);\draw (2.65,3)--(4.75,1);
				\draw(1,1)--(5,1);
			\end{tikzpicture}
		\end{center}
		\caption{$\mathscr{F}_{\infty}$}
		\label{fig:famFinf}
	\end{subfigure}
	\begin{subfigure}[b]{.47\textwidth}
		\begin{center}
			\begin{tikzpicture}
				[scale=1,auto=left,every node/.style={circle, fill, inner sep=0 pt, minimum size=.75mm, outer sep=0pt},line width=.8mm]
				\node(1) at (2.65,3){};
				\node (3) at (1,1){};
				\node (3a) at (1.5,1){};
				\node (3b) at (1.25,1.5){};
				\node (4) at (2.25,1){};
				\node (4b) at (2.5,1.5){};
				\node (4a) at (2.75,1){};
				\node(5) at (3.5,1){};
				\node (5a) at (4,1){};
				\node (5b) at (3.75,1.5){};
				\node (6) at (4.75,1){};
				\node (6a) at (5.25,1){};
				\node (6b) at (5,1.5){};
				\draw[dotted, line width=.4mm] (5.45,2)--(5.75,2);
				\draw (2.65,3)to (1.25,1.5); \draw (2.65,3)--(2.5,1.5);\draw(2.65,3)--(3.75,1.5);\draw (2.65,3)--(5,1.5);
				\draw (1.5,1)--(2.25,1); \draw (2.75,1)--(3.5,1); \draw (4,1)--(4.75,1); 
				\draw[line width=.4mm] (1,1)--(1.5,1)-- (1.25,1.5)--(1,1);
				\draw[line width=.4mm] (2.25,1)--(2.75,1)--(2.5,1.5)--(2.25,1);
				\draw[line width=.4mm] (3.5,1)--(4,1)--(3.75,1.5)--(3.5,1);
				\draw[line width=.4mm] (4.75,1)--(5.25,1)--(5,1.5)--(4.75,1);
				\draw (5.25,1)--(5.5,1);
			\end{tikzpicture}
		\end{center}
		\caption{$\mathscr{F}_{\infty}^{\Delta}$}
		\label{fig:Fdeltinf}
	\end{subfigure}
	\caption{Two types of fan-like structures}
	\label{fig:fans}
\end{figure}

Let $L_\infty$ be the graph obtained from two disjoint rays $p_1p_2...$ and $q_1q_2...$ by joining $p_i$ with $q_i$ for all $i$;
we call the rays \emph{rails}, let $\mathscr L_\infty$ be the family of graphs obtained from $L_\infty$ by subdividing each $p_iq_i$ at least once and each of the other edges an arbitrary number, possibly zero, of times; see \cref{fig:L1inf}. 
Let $L_\infty^\Delta$ be obtained from $L_\infty$ by first subdividing each $p_iq_i$ and  $q_iq_{i+1}$ by $s_i$ and $t_i$, respectively, and then joining $s_i$ with $t_i$ for all $i$; let $\mathscr L_\infty^\Delta$ be the family of graphs obtained from $L_\infty^\Delta$ by subdividing each edge that is not in any triangle, an arbitrary number, possibly zero, of times; see \cref{fig:L2inf}. 

Let $L_\infty^{\nabla\Delta}$ be obtained from $L_\infty^\Delta$ as follows. We first subdivide each $p_ip_{i+1}$ by $t_i'$. Then, for each $i$, let $s'_i$ be either $s_i$ or a vertex subdividing $p_is_i$. Finally, we add edges $s'_it'_i$ for all $i$. Let $\mathscr L_\infty^{\nabla\Delta}$ be the family of graphs obtained from $L_\infty^{\nabla\Delta}$ by subdividing each edge that is not in any triangle, an arbitrary number, possibly zero, of times; see \cref{fig:L3inf}. Note that some of the vertical paths could be trivial, by which we mean that the length of the path is zero. In such a case, the path is a vertex and this vertex has degree four in $L_\infty^{\nabla\Delta}$.

Graphs as illustrated in \cref{fig:infcl} are infinite analogs of graphs in $\Lambda_n$. 
An \emph{infinite slim ladder} is a locally finite graph that consists of two disjoint rays $p_1p_2...$ and $q_1q_2...$ together with an infinite set of edges of the form $p_iq_j$ (including $p_1q_1$) such that for any two such edges $p_iq_j$ and $p_{i'}q_{j'}$ we always have $(i-i')(j-j')\ge -1$. In other words, either $p_iq_j$ and $p_{i'}q_{j'}$ do not cross each other (when $(i-i')(j-j')\ge 0$) or they cross in the closest possible way (when $(i-i')(j-j')=-1$). 

\begin{figure}[h!]
	\begin{subfigure}[b]{.23\textwidth}
		\begin{center}
			\begin{tikzpicture}
				[scale=.6,auto=left,every node/.style={circle, fill, inner sep=0 pt, minimum size=.75mm, outer sep=0pt},line width=.8mm]
				\node (1) at (1,1){};\node (2) at (1,3){};
				\draw (1.center)--(2.center);
				\draw (1,1)--(5.6,1);
				\draw(1,3)--(5.6,3);
				\draw (2.25,1)--(2.25,3);
				\draw (3.5,1)--(3.5,3);
				\draw (4.75,1)--(4.75,3);
				\draw[dotted, line width=.4mm] (5.05,2)--(5.45,2);	
			\end{tikzpicture}
		\end{center}
		\caption{$\Loneinf$}
		\label{fig:L1inf}
	\end{subfigure}
	\begin{subfigure}[b]{.23\textwidth}
		\begin{center}
			\begin{tikzpicture}
				[scale=.6,auto=left,every node/.style={circle, fill, inner sep=0 pt, minimum size=.75mm, outer sep=0pt},line width=.8mm]
				\node(1) at (1.25,3){};
				\node (3) at (1,1){};
				\node (3a) at (1.5,1){};
				\node (3b) at (1.25,1.5){};
				\node (4) at (2.25,1){};
				\node (4b) at (2.5,1.5){};
				\node (4a) at (2.75,1){};
				\node(5) at (3.5,1){};
				\node (5a) at (4,1){};
				\node (5b) at (3.75,1.5){};
				\node (6) at (4.75,1){};
				\node (6a) at (5.25,1){};
				\node (6b) at (5,1.5){};
				\draw[dotted, line width=.4mm] (5.45,2)--(5.75,2);
				\draw (1)to(1.25,1.5); 
				\draw (1.5,1)--(2.25,1); \draw (2.75,1)--(3.5,1); \draw (4,1)--(4.75,1); 
				\draw[line width=.4mm] (1,1)--(1.5,1)--(1.25,1.5)--(1,1);
				\draw[line width=.4mm] (2.25,1)--(2.75,1)--(2.5,1.5)--(2.25,1);
				\draw[line width=.4mm] (3.5,1)--(4,1)--(3.75,1.5)--(3.5,1);
				\draw[line width=.4mm] (4.75,1)--(5.25,1)--(5,1.5)--(4.75,1);
				\draw(5.25,1) --(6,1);
				\draw(1.center)--(6,3);
				\draw (1.25,1.5)--(1.25,3); \draw (2.5,1.5)--(2.5,3);\draw (3.75,1.5)--(3.75,3);\draw (5,1.5)--(5,3);
			\end{tikzpicture}
		\end{center}
		\caption{$\Ltwoinf$}
		\label{fig:L2inf}
	\end{subfigure}
	\begin{subfigure}[b]{.23\textwidth}
		\begin{center}
			\begin{tikzpicture}
				[scale=.6,auto=left,every node/.style={circle, fill, inner sep=0 pt, minimum size=.75mm, outer sep=0pt},line width=.8mm]
				\node (3) at (1,1){};
				\node (3a) at (1.5,1){};
				\node (3b) at (1.25,1.5){};
				\node (4) at (2.25,1){};
				\node (4b) at (2.5,1.5){};
				\node (4a) at (2.75,1){};
				\node(5) at (3.5,1){};
				\node (5a) at (4,1){};
				\node (5b) at (3.75,1.5){};
				\node (6) at (4.75,1){};
				\node (6a) at (5.25,1){};
				\node (6b) at (5,1.5){};
				\node (7) at (1,3){};
				\node (7a) at (1.5,3){};
				\node (7b) at (1.25,2.5){};
				\node (8) at (2.25,3){};
				\node (8b) at (2.5,2.5){};
				\node (8a) at (2.75,3){};
				\node(9) at (3.5,3){};
				\node (9a) at (4,3){};
				\node (9b) at (3.75,2.5){};
				\node (10) at (4.75,3){};
				\node (10a) at (5.25,3){};
				\node (10b) at (5,2.5){};
				\draw[dotted, line width=.4mm] (5.45,2)--(5.75,2);
				\draw (1.5,1)--(2.25,1); \draw (2.75,1)--(3.5,1); \draw (4,1)--(4.75,1); 
				\draw[line width=.4mm] (1,1)--(1.5,1)--(1.25,1.5)--(1,1);
				\draw[line width=.4mm] (2.25,1)--(2.75,1)--(2.5,1.5)--(2.25,1);
				\draw[line width=.4mm] (3.5,1)--(4,1)--(3.75,1.5)--(3.5,1);
				\draw[line width=.4mm] (4.75,1)--(5.25,1)--(5,1.5)--(4.75,1);
				\draw (5.25,1) --(6,1);
				\draw (1.25,1.5)--(1.25,2.5); \draw (2.5,1.5)--(2.5,2.5);\draw (3.75,1.5)--(3.75,2.5);\draw (5,1.5)--(5,2.5);
				\draw (1.5,3)--(2.25,3); \draw (2.75,3)--(3.5,3); \draw (4,3)--(4.75,3); 
				\draw[line width=.4mm] (1,3)--(1.5,3)--(1.25,2.5)--(1,3);
				\draw[line width=.4mm] (2.25,3)--(2.75,3)--(2.5,2.5)--(2.25,3);
				\draw[line width=.4mm] (3.5,3)--(4,3)--(3.75,2.5)--(3.5,3);
				\draw[line width=.4mm] (4.75,3)--(5.25,3)--(5,2.5)--(4.75,3);
				\draw  (5.25,3)--(6,3);
			\end{tikzpicture}
		\end{center}
		\caption{$\Lthreeinf$}
		\label{fig:L3inf}
	\end{subfigure}
	\begin{subfigure}[b]{.23\textwidth}
		\begin{center}
			\begin{tikzpicture}
				[scale=.6,auto=left,every node/.style={circle, fill, inner sep=0 pt, minimum size=.75mm, outer sep=0pt}, line width=.4mm]
				\node (1) at (1,1){};
				\node (2) at (1,3){};
				\draw[line width=.8mm](1,1) to (4.05,1); \draw (4.05,1)--(4.6,1);\draw[line width=.8mm, line cap=round] (4.6,1)--(5.6,1); \draw[line width=.8mm](1,3) to (4.05,3);\draw (4.05,3)--(4.6,3);\draw[line width=.8mm]  (4.6,3) to (5.6,3);\draw(1,1)--(1,3);
				\draw (1.75,3) --(1.25,1);\draw (1.75,3)--(1.75,1);\draw(1.75,3)--(2.25,1);
				\draw (2.5,3) to (2.5,1);\draw (3.5,1)--(3.5,3);
				\draw (4.05,3) to (4.6,1);\draw(4.6,3) to (4.05,1);
				\draw[dotted] (4.95,2)--(5.35,2);
			\end{tikzpicture}
		\end{center}
		\caption{Infinite Slim Ladder}
		\label{fig:infcl}
	\end{subfigure}
	\caption{Four types of ladder-like structures}
	\label{fig:ladders}
\end{figure}

The main result of the paper is stated below.  
\begin{restatable}{theorem}{thminftwocon}\label{thm:infinite}
Let $G$ be a $2$-connected infinite graph. 
Then $G$ contains one of the following as an induced subgraph: $K_{\infty}$, an infinite slim ladder, or a member of $\Theta_\infty\cup\TFinf\cup\mathscr{F}_{\infty}^{\Delta}\cup\Loneinf\cup\Ltwoinf\cup\Lthreeinf$.
\end{restatable}

	To make a clear distinction between vertices and edges, we will use subscripts on lowercase letters for vertices such as $v_i$ and superscripts on lowercase letters for edges such as $e^i$.	
 
	To prove the main theorem, we consider the cases that an infinite 2-connected graph $G$ either has a vertex of infinite degree or it does not.
	 \cref{sec:infvertex} discusses the case that $G$ has a vertex of infinite degree.  
	 The main challenge in going from finite graphs to infinite graphs is in that section.
	 In \cref{sec:locfin}, we start with a locally-finite graph and obtain one of the four ladder-like structures from \cref{fig:ladders}.
	 In \cref{sec:provingmainthm}, we combine the results of \cref{sec:infvertex,sec:locfin} to prove \cref{thm:infinite}, and we present an alternate proof to \cref{thm:finite} using \cref{thm:infinite}.
	 
\section{Vertex of Infinite Degree}\label{sec:infvertex}
In this section, we present the unavoidable induced subgraphs of an infinite 2-connected graph with a vertex of infinite degree.

Next, we will define notation to make it easier to refer to specific subpaths.
The vertices of a ray $P$ with initial vertex $p$ have a natural ordering:
we write $u\le v$ whenever $u$ lies on the $p, v$-subpath of~$P$.
We write $u < v$ whenever $u\leq v$ and $u\neq v$.
For a ray $P$, let $P[u,v]$ be the $u,v$-subpath of $P$, let $P[u,v)$ be the subpath of $P$ that consists of vertices $u\le x<v$, let $P(u,v)$ be the subpath of $P$ that consists of vertices $u<x<v$, let $P[u,\infty)$ be the subray of $P$ has initial vertex $u$, and let $P(u,\infty)$ be the subray of $P$ that consists of vertices $u<x<\infty$.
For the following lemmas, we need to define a tree $T$ that controls the edges of an infinite connected graph $G$ that are not in $T$ and whose endpoints are in $T$.

\begin{definition}\label{def:vconntree}
Let $\mathcal{V}$ be an infinite independent set of vertices of a connected graph $G$.
A \emph{$\mathcal{V}$-connecting tree} is a subgraph $T$ of $G$ defined through the following inductive process.
Let $v_1$ be an arbitrary element of $\mathcal{V}$, and let $P^1$ be the graph consisting of the single vertex, $v_1$. 
To be consistent with all the other $P^i$s defined below, we consider $P^1$ as a path with endpoints $t_1$ and $v_1$ where $t_1=v_1$.
Let $P^2$ be the shortest path in $G$ from $v_1$ to a vertex in $\mathcal{V}\setminus \{v_1\}$.  
Let $v_2$ be the endpoint of $P^2$ in $\mathcal{V}\setminus\{v_1\}$ and $t_2=v_1$. 
Suppose now that $k$ is a natural number and that $v_i$, $t_i$, and $P^i$ have been defined for all $i$ in $\{1,2,\dots,k\}$.
Every path $Q$ joining a vertex in $\mathcal{V}\setminus \{v_1,v_2,\dots,v_k\}$ to a vertex of $P^1\cup P^2\cup\dots \cup P^k$ will receive a triple $(\gamma_1(Q),\gamma_2(Q),\gamma_3(Q))$ as a \emph{grade}.  
In this triple, let $\gamma_1(Q)$ be the length of $Q$, let $\gamma_2(Q)$ be the minimum $i$ such that the endpoint $u$ of $Q$ lies on $P^i$, and let $\gamma_3(Q)$ be the length of $P^i[u,t_i]$. 
Let $P^{k+1}$ be the path with lexicographically minimal grade.  
Let the endpoint of $P^{k+1}$ on $P^i$ be called $t_{k+1}$ and the endpoint of $P^{k+1}$ in $\mathcal{V}\setminus \{v_1,v_2,\dots, v_k\}$ be called $v_{k+1}$.
Let $T$ be the union of all the paths $P^k$.
\end{definition}

We have the following observations about $T$.
\begin{observation}\label{obs1} \hfill

\begin{statement}
	\item At each step, it is always possible for us to select a path because $G$ is connected. \label{obs1i}
	\item \label{obs1iii} At the $k$-th step for $k>1$ in \cref{def:vconntree}, we added a path $P^k$ whose one endpoint, $t_k$ is on $P^1\cup P^2\cup \dots \cup P^{k-1}$ and is otherwise disjoint from $P^1\cup P^2\cup \dots\cup P^{k-1}$.
 \item \label{obs1ii} By \ref{obs1iii}, $T$ is a tree.
	\item \label{obs1iv}The tree $T$ is a union of induced paths in $G$, but $T$ is not necessarily induced in $G$.
	\item \label{obs1v}Each vertex $v_k$ belongs to $\mathcal{V}$, and no internal vertices of $P^k$ belong to $\mathcal{V}$.
	\item \label{obs1vi}Let $u$ be the neighbor of $t_k$ on $P^k$. Then $G$ has no edge between $P^k-t_k-u$ and $P^1\cup \dots \cup P^{k-1}$.
	\item \label{obs1vii}Let $u$ be the neighbor of $t_k$ on $P^k$, then for any $j<k$, the set $N_G(u)\cap V(P^j)$ can be covered by a subpath of $P^j$ on at most two edges.
\end{statement}
\end{observation}

For \cref{lem:treeinfvert,lem:treelocfin}, we assume that $G$ is an infinite 2-connected graph with a vertex of infinite degree $v^*$ and no $K_{\infty}$ as an induced subgraph, that $\mathcal{V}$ is an infinite independent subset of the neighborhood of $v^*$ in $G$, that $T$ is a $\mathcal{V}$-connecting tree of $G-v^*$, and that the notation used is that from \cref{def:vconntree}.
Since $T$ is infinite and connected, \cref{thm:konig} implies that $T$ contains a vertex of infinite degree or a ray.
We address the case that $T$ has a vertex of infinite degree in \cref{lem:treeinfvert} and the case that $T$ contains a ray in \cref{lem:treelocfin}.

\begin{lemma}\label{lem:treeinfvert}
	If $T$ has a vertex of infinite degree, then $G$ contains as an induced subgraph a member of the family $\Theta_{\infty}$.
\end{lemma}
\begin{proof}
	Since $T$ has a vertex of infinite degree, it follows that $t_i$ is the same vertex of $G$ for an infinite sub-sequence $\mathcal{I}$ of natural numbers; call this vertex $x$.
	For each $i\in\mathcal{I}$, let $w_i$ be the neighbor of $v^*$ on $P^i(x,v_i]$ such that the length of $P^i(x,w_i]$ is minimal.
    Let $x_i$ be the member of $N_T(x)$ on the path $P^i$, let $\mathcal{I}_1$ be an infinite sub-sequence of $\mathcal{I}$ such that $\{x_i~:~i\in \mathcal{I}_1\}$ is stable in $G$, and let $\mathcal{P}=\{P^i[x,w_i]~:~i\in\mathcal{I}_1\}$.
	Let $T^1$ be the sub-tree of $T$ consisting of the paths in $\mathcal{P}$.
	Note that $T^1$ is a subdivision of $K_{1,\infty}$ and that $T^1$ is not necessarily induced in $G$.
	
	Suppose that $x_i$ is a neighbor of $v^*$ for infinitely many $i$ in $\mathcal{I}_1$. 
	Let $\mathcal{I}_{2a}$ be the sub-sequence of $\mathcal{I}_1$ of indices $i$ such that $x_i$ is a neighbor of $v^*$.
	Then the subgraph of $G$ induced by the vertices $v^*$, $x$, and $x_i$ for all $i\in\mathcal{I}_{2a}$ is a member of the family $\Theta_{\infty}$ where each path at most two edges, and the conclusion follows.
	
	We may therefore assume that $x_i$ is not a neighbor of $v^*$ for infinitely many $i$ in $\mathcal{I}_1$; and thus $P^i[x_i,w_i]$ is non-trivial.
	Let $\mathcal{I}_{2b}$ be the sub-sequence of $\mathcal{I}_1$ consisting of $i$ for which $P^i[x_i,w_i]$ is not trivial, let $\mathcal{P}^{2b}$ be the sub-sequence of $\mathcal{P}$ that consists of paths $P^j[x,w_j]$ for $j$ in $\mathcal{I}_{2b}$, and let $T^{2b}$ be the sub-tree of $T^1$ that consists of paths of $\mathcal{P}^{2b}$.
	Let $N_1=\{x_i~|~i\in\mathcal{I}_{2b}\}$, and let $N_2=\{y_j~|~P^j[x,y_j]$ has length $2\}$
 be the set of vertices of $T^{2b}$ that are distance two from $x$ in $T^{2b}$.
	By \crefWithTheorem{obs1vi} and \ref{obs1vii}, if $e$ is an edge of $G[V(T^{2b})]$ not in $T^{2b}$, then $e=x_iy_j$ with $j<i$.   
	
	Let $H$ be the subgraph of $G$ induced by the vertices of $N_1\cup N_2$.
	Note that $H$ is infinite and bipartite.
	We will consider two cases, either $H$ has a vertex of infinite degree or not.
	If $H$ has a vertex of infinite degree, say $u$, then $u\in N_2$, otherwise this would imply that $j>i$ for some edge $x_iy_j$.
 
	Suppose $H$ has a vertex of infinite degree, say $u$. 
	Let $N_1'$ be the subset of $N_1$ that consists of vertices $x_i$ such that $x_i\in N_1$ is a neighbor of $u$.
	Note that $N_1'$ is an independent set of vertices. 
	The subgraph of $G$ induced by $u$, $x$, and the vertices of $N_1'$ is $K_{2,\infty}$ and thus is a member of the family $\Theta_{\infty}$, and the conclusion follows. 
	
	We may therefore assume that $H$ is locally-finite.
    Note that $H$ has an infinite matching $M$ consisting of edges of the form $x_i y_i$ for $i$ in $\mathcal{I}_{2b}$.
    $M$ has an infinite subset $M'$ that is an induced matching selected by a greedy process in $G$ because $H$ is locally-finite.
    Let $\mathcal{J}$ be the set of indices $i$ for which $x_i y_i$ is an edge of $M'$.
    The graph induced by $v^*$, and the vertices of $P^j[x,w_j]$ for $j$ in $\mathcal{J}$ is a member of the family $\Theta_{\infty}$; and the conclusion follows.
\end{proof}


We have considered the case that $T$ has a vertex of infinite degree. Next, we address the case that $T$ is locally-finite.

\begin{lemma}\label{lem:treelocfin}
If $T$ is locally-finite, then $G$ contains a member of $\mathscr F_\infty\cup \mathscr F_\infty^\Delta$ as an induced subgraph. 
\end{lemma}

\begin{proof}
Since $T$ is locally-finite, \cref{thm:konig} implies that $T$ contains a ray $R$ as a subgraph. Let $\mathcal I=\{i: E(P^i\cap R)\ne\emptyset\}$. Then $\cal I$ is infinite as each $P^i$ is finite. By taking a sub-ray of $R$, if necessary, we assume that the first edge of $R$ belongs to $P^{i_0}$, where $i_0=\min\{i:i\in\mathcal I\}$. Let $T'=\bigcup\limits_{i\in \mathcal{I}} P^i$. To understand the structure of $T'$, we make the following observations. 

Let $R=r_1r_2\ldots$ and let us label each edge $e_n=r_nr_{n+1}$ by $\ell(e_n):=i$, where $e_n\in E(P^i)$. In particular, $\ell(e_1)=i_0$. We claim that the sequence $\ell(e_1), \ell(e_2),\ldots$ is non-decreasing. Suppose otherwise. Then there exists the smallest index $n$ with $\ell(e_n)> \ell(e_{n+1})$. By the minimality of $\ell(e_1)$, we have that
$\ell(e_n)\ne i_0$ and 
there exists $1<m\le n$ with $\ell(e_{m-1})<\ell(e_m) = \ldots = \ell(e_n)$. This means $P^i$, where $i=\ell(e_n)$, intersects $P^1\cup \cdots \cup P^{i-1}$ on at least two vertices ($r_m$ and $r_{n+1}$), which contradicts \crefWithTheorem{obs1iii}, and thus it proves our claim. 
This claim implies that each $P^i\cap R$ is a path, which allows us to assume, without loss of generality, that $t_{i_0}$ is the initial vertex of $R$, see \cref{fig:rayinT}. 
Together with \crefWithTheorem{obs1iii}, our claim also implies that, if we consider every instance the label increases, $\ell(e_n) < \ell(e_{n+1}) = i$, then $r_{n+1}$, the common vertex of $e_n$ and $e_{n+1}$, is $t_i$. 
Since $\mathcal{I}$ is a set, we need to define the next term as
as $i^+=\min\{k\in\mathcal I: k>i\}$ for each $i\in\cal I$. 
Then $P^i\cap R= P^i[t_i,t_{i^+}]$.
Thus, $R =\bigcup\limits_{i\in\mathcal{I}}P^i[t_i,t_{i^+}]$ and $T$ is the union of $R$ and $\bigcup\limits_{i\in\mathcal{I}}P^i[t_{i^+},v_i]$. 

\begin{figure}[ht]
\begin{subfigure}[b]{.6\textwidth}
\begin{center}
			\begin{tikzpicture}
				[scale=.5,auto=left,every node/.style={circle, fill, inner sep=0 pt, minimum size=2mm}, line width=1mm]
				\node (1) at (1.75,2) [label=above:$t_{i_0}$]{};
				\node[fill=red] (2) at (5,1) [label=below:$v_{i_0}$]{};
				\node[fill=red] (4) at (13,1) {};
				\node[fill=red] (5) at (17,1)[label=below:$v_i$]{};
				\node[fill=red] (6) at (21,1) [label=below:$v_{i^+}$] {};
				\node (7) at (4.3,2.5) {};
				\node[fill=red] (8) at (8.5,2.5) [label=above:$v\eq t$] {};
				\node (9) at (12.4,2.5)[label=above:$t_i$] {};
				\node (10) at (16.4,2.5)[label=above:$t_{i^+}$]{};
				\node (11) at (20.3,2.5){};
				\begin{scope}[on background layer]
				\draw[bend left=50, line cap=round,blue, line width=.8mm] (1.center) to (7.center);
				\draw[bend left=20, line cap=round,line width=.8mm] (7.center) to (2.center);
				\draw[bend left=50, line cap=round,line width=.8mm,blue] (7.center) to (8.center);
				\draw[bend left=50, line cap=round,blue,line width=.8mm] (8.center) to (9.center);		
				\draw[bend left=20, line cap=round,line width=.8mm] (9.center) to (4.center);
				\draw[bend left=50, line cap=round,blue,line width=.8mm] (9.center) to node [midway, above, fill=white] {$\textcolor{black}{P^i}$} (10.center);
				\draw[line cap=round, bend left=20, line width=.8mm] (10.center) to (5.center);
				\draw[line cap=round, bend left=50,blue, line width=.8mm] (10.center) to  node [midway, above, fill=white] {$\textcolor{black}{P^{i^+}}$} (11.center);
				\draw[line cap=round, bend left=20, line width=.8mm] (11.center) to (6.center);
				\draw[line cap=round, bend left=30,blue, line width=.8mm] (11.center) to node[near end, above, fill=white]{$R$}(21.9,3.1);
    \end{scope}
			\end{tikzpicture}	
		\end{center}
  \caption{$T'$}
  \label{fig:rayinT}
\end{subfigure}
\begin{subfigure}[b]{.35\textwidth}
\begin{center}	
		\begin{tikzpicture}					
			[scale=.5,auto=left,every node/.style={circle, fill, inner sep=0 pt, minimum size=2mm}, line width=.8mm]
    \node (ti) at (1,3.75) [label=above:$t_i$]{};
    \node (ti+) at (4,3.75) [label=above:$t_{i^+}$]{};
    \node (xi+) at (5.25,3.75) [label={[label distance=-2pt]above:$x_{i^+}$}]{};
    \node (tj) at (8,3.75) {};
    \node[minimum size=1.5mm] (yi) at (4,2.75) [label=left:$y_i$]{};
    \node[minimum size=1.5mm] (zi) at (4,1.75) [label=right:$z_i$]{};
    \node[fill=red] (vi) at (4,1) [label=below:$v_i$]{};
    \node[fill=red](vi+) at (8,1.75) [label=below:$v_{i^+}$]{};
    \begin{scope}[on background layer]
        \draw[line width=.8mm, line cap=round, blue] (ti.center) to (ti+.center);
        \draw[line width=.4mm, blue] (ti+.center) to (xi+.center);
        \draw[line width=.8mm,blue] (xi+.center) to (tj.center);
        \draw[line width=.4mm] (ti+.center) to (yi.center) to (zi.center);
        \draw[line width=.8mm] (zi.center) to (vi.center);
        \draw[LightGray, line width=.4mm] (yi.center) to (xi+.center);
        \draw[DarkGray, line width=.4mm] (zi.center) to (xi+.center);
        \draw[line width=.8mm] (tj.center) to (vi+.center);
    \end{scope}
   \end{tikzpicture}
   \end{center}

	\caption{Potential Edges of $G$ not in $T'$}
	\label{fig:sidestep}
 \end{subfigure}
 \caption{}
 \label{fig:schematicinfvert}
	\end{figure}
In general, $T'$ is not an induced subgraph of $G$. 
However, all extra edges are of a special type. 
For each $i\in\mathcal I$, let $x_{i^+}$ and $y_i$ be the neighbors of $t_{i^+}$ on $P^{i+}$ and $P^i[t_{i^+},v_i]$, respectively; let $z_i$ be the neighbor of $y_i$ on $P^i[y_i,v_i]$. 
Note that $y_i$ (or $z_i$) does not exist if the length of $P^i[t_{i^+}, v_i]$ is 0 (or $\le 1$). Since $T'$ is obtained from $P^{i_0}$ by repeatedly adding paths $P^i$ ($i\in\mathcal I-i_0$) along $R$, we conclude from \crefWithTheorem{obs1vi} and \ref{obs1vii} and the minimality of $\gamma_1(P^i)$ (see \cref{def:vconntree}) that every edge of $G$ not in $T'$ but with both ends in $T'$ must be $y_ix_{i^+}$ or $z_ix_{i^+}$ for some $i\in\mathcal I$, see \cref{fig:sidestep}. In particular, $R$ is an induced ray of $G$.

We can now find a desired induced subgraph in $G[V(T')\cup\{v^*\}]$. For each $i\in\cal I$, let $w_i\in N_G(v^*)$ be the vertex on $P^i(t_i,v_i]$ with $P^i[t_i,w_i]$ minimal. Such $w_i$ must exist since $v_i\in N_G(v^*)$. If $R$ contains infinitely many $w_i$, then, for a sub-ray $R'$ of $R$, the subgraph of $G$ induced by $v^*$ and the vertices of $R'$ is a member of $\mathscr F_\infty$. We may therefore assume, without loss of generality, that no $w_i$ is contained in $R$. 
This means that every $w_i$ is contained in $P^i(t_{i^+},v_i]$. 
For each $i\in\mathcal I$, let $S^i=P^i[t_{i^+},w_i]\cup \{w_iv^*\}$. Then $S_i$ is an induced path of $G$ between $v^*$ and $t_{i^+}$.
We partition $\mathcal{I}$ based on $V(S^i)\cap N_G(x_{i^+})$. 
Since this intersection always contains $t_{i^+}$, it is non-empty for each $i\in\mathcal{I}$.   
Let 
\begin{itemize}
\item $\mathcal I_1=\{i\in\mathcal I: V(S^i)\cap N_G(x_{i^+})=\{t_{i^+}\}\}$,
\item $\mathcal I_2=\{i\in\mathcal I: V(S^i)\cap N_G(x_{i^+})=\{t_{i^+},y_i\}\}$, and
\item $\mathcal I_3=\{i\in\mathcal I: V(S^i)\cap N_G(x_{i^+})\supseteq\{t_{i^+}, z_i\}\}$. 
\end{itemize}
Since $(\mathcal I_1, \mathcal I_2, \mathcal I_3)$ forms a partition of $\cal I$ and $\mathcal{I}$ is infinite, it follows that at least one of $\mathcal I_1, \mathcal I_2, \mathcal I_3$ is infinite. 
If $\mathcal I_1$ is infinite, then, let $S=\bigcup \{S^i~:~i\in\mathcal{I}_1\}$, and for a sub-ray $R'$ of $R$,  
the subgraph of $G$ induced by the vertices of $R'$ and $S$ is a member of the family $\TFinf$, as desired
If $\mathcal I_2$ is infinite, then we consider the elements of $\mathcal{I}_2$ as an increasing sequence.
Let $\mathcal I'_2$ be the sub-sequence obtained by taking every other element of $\mathcal I_2$, and let $S=\bigcup\{S_i:i\in\mathcal I_2'\}$. Then, for a sub-ray $R'$ of $R$, the subgraph of $G$ induced by the vertices of $R$ and $S$ is a member of the family $\mathscr{F}^{\Delta}$, as desired.
We may therefore assume that $\mathcal{I}_3$ is infinite.  Note that $y_i$ may be in $S^i$.
Let $S=\bigcup\{S^i-y_i~:~i\in\mathcal{I}_3\}$.
Then, for a sub-ray $R'$ of $R$, the subgraph of $G$ induced by the vertices of $R$ and $S$ is a member of the family $\mathscr{F}_\infty$, as desired.
\end{proof}	

Now, we will describe the unavoidable infinite induced subgraphs of a 2-connected graph that has a vertex of infinite degree.
\begin{lemma}\label{lem:infvert}
	Let $G$ be a 2-connected graph with a vertex of infinite degree.  Then $G$ contains as an induced subgraph either $K_{\infty}$ or a member of $\Theta_{\infty}\cup\TFinf\cup\TFinf^{\Delta}$. 
\end{lemma}

\begin{proof}
	Let $v^*$ be a vertex of infinite degree of $G$ and let $N_G(v^*)$ be the neighborhood of $v^*$.  
By \cref{thm:inframsey}, the graph induced by the vertices of $N_G(v^*)$, contains either $K_{\infty}$, and the conclusion follows, or $\overline{K}_\infty$.
Let $\mathcal{V}$ be the subset of $N_G(v^*)$ that induces $\overline{K}_{\infty}$.

Let $T$ be a $\mathcal{V}$-connecting tree of $G$ as described in \cref{def:vconntree}.
By \cref{thm:konig}, $T$ has either a vertex of infinite degree or a ray.
If $T$ has a vertex of infinite degree, then \cref{lem:treeinfvert} implies that $G$ contains as an induced subgraph a member of the family $\Theta_{\infty}$, and the conclusion follows.	
We may therefore assume that $T$ is locally-finite.
\cref{lem:treelocfin} implies that $G$ contains as an induced subgraph a member of $\TFinf \cup\TFinf^{\Delta}$, and the conclusion follows.
\end{proof}

Note that \cref{lem:infvert} is no longer true if we require the unavoidable graph to contain a prescribed vertex of infinite degree. 
For instance, let $H=K_{\infty}$, let $v^*$ be a specified vertex of $H$, and let $G$ be obtained from $H$ by subdividing at least once every edge that is incident to $v^*$.  
No induced subgraph of $G$ belongs to $\TFinf\cup \TFinf^{\Delta}\cup \Theta_{\infty}$ and no infinite clique of $G$ contains $v^*$.

\section{Locally Finite Graph}\label{sec:locfin}
In this section, we determine the unavoidable induced subgraphs of an infinite locally finite 2-connected  graph. We have seen in \cref{thm:konig} that every infinite locally finite connected graph must contain a ray. It turns out that the unavoidable graphs we are looking for are constructed from these rays. In order to identify the most relevant rays, we first need some basic properties of rays, which we present below. 

For any two rays $R_1$ and $R_2$ of a graph $G$, we write $R_1\sim R_2$ if $G$ has infinitely many disjoint paths joining the two rays. It is well known \cite{Halin2} that $R_1\sim R_2$ if and only if, for any finite $Z\subseteq V(G)$, the infinite component of $R_1-Z$ and the infinite component of $R_2-Z$ are contained in the same component of $G-Z$. In addition, $\sim$ is an equivalence relation and each equivalence class is called an \emph{end} of $G$. It is worth noting that all subrays of a ray belong to the same end.

To display all rays in an end, Halin introduced in \cite{Halin2} a special graph structure, which we call a \emph{Halin decomposition} of a graph $G$, see \cref{fig:halindecomp}. It is a sequence $\{G_i : i \in\mathbb N\}$ of subgraphs of $G$ satisfying 
\begin{enumerate}[label=(H\arabic*),leftmargin=15mm,itemsep=0pt,topsep=0pt]
\item\label{H1}  $G=\cup\{G_i:i\in\mathbb N\}$,
\item\label{H2}  $G_i \cap G_j = \emptyset$; whenever $|j-i|\ge 2$,
\item\label{H3}  for each $i\in\mathbb N$, $S_i := V(G_i \cap G_{i+1})$ is finite, and 
\item\label{H4}  for each $i\in\mathbb N$, $G_{i+1}$ has a set, $\mathcal W_{i+1}$, of $|S_i|$ disjoint paths between $S_i$ and $S_{i+1}$.
\end{enumerate}
It is worth noting, by \ref{H2}, that $S_1, S_2, \ldots$ are pairwise disjoint and, by \ref{H1} and \ref{H2}, that 
\begin{enumerate}[label=(H\arabic*),resume,leftmargin=15mm,itemsep=0pt,topsep=0pt]
\item\label{H5}  for each $i\in\mathbb N$, $A_i:=V(G_1\cup \ldots \cup G_i)-S_i$ and $B_i:=V(G_{i+1}\cup G_{i+2}\cup \ldots)-S_i$ form a partition of $V(G)-S_i$ such that $G$ has no edge between $A_i$ and $B_i$.
\end{enumerate}

\begin{figure}[htb]
\begin{center}	
    \begin{tikzpicture}					
        [scale=.85,auto=left,every node/.style={circle, fill, inner sep=0 pt, minimum size=2mm}, line width=.5mm]
       \node (s11) at (1,1) {};
       \node (s12) at (1,1.5){};
       \node (s13) at (1,2){};
       \node (s14) at (1,2.5)[label={[label distance=6pt]above:$S_1$}]{};
       \node (s21) at (3, .5){};
       \node (s22) at (3,1){};
       \node (s23) at (3,1.5){};
       \node (s24) at (3,2){};
       \node (s25) at (3,2.5){};
       \node (s26) at (3,3)[label={[label distance=6pt]above:$S_2$}]{};
       \node (s31) at (5,.5){};
       \node (s32) at (5,1){};
       \node (s33) at (5,1.5){};
       \node (s34) at (5,2){};
       \node (s35) at (5,2.5){};
       \node (s36) at (5,3){};
       \node (s37) at (5,3.5)[label={[label distance=6pt]above:$S_3$}]{};
       \node (s41) at (7,.5){};
       \node (s42) at (7,1){};
       \node (s43) at (7,1.5){};
       \node (s44) at (7,2){};
       \node (s45) at (7,2.5){};
       \node (s46) at (7,3){};
       \node (s47) at (7,3.5){};
       \node (s48) at (7,4)[label={[label distance=6pt]above:$S_4$}]{};
       \draw (-1.25,.75) to  (-1.25,2.75);
       \draw[bend right=30] (-1.25,.75) to (.75,.75);
       \draw[bend left=30] (-1.25,2.75) to node[midway, above,fill=white] {$G_1$} (.75,2.75);
       \draw[bend right=30] (1.25,.75) to (2.75,.25);
       \draw[bend left=30] (1.25,2.75) to node[midway, above,fill=white] {$G_2$} (2.75,3.25);
        \draw[bend right=30] (3.25,.25) to (4.75,.25);
       \draw[bend left=30] (3.25,3.25) to node[midway, above,fill=white] {$G_3$} (4.75,3.75);
        \draw[bend right=30] (5.25,.25) to (6.75,.25);
       \draw[bend left=30] (5.25,3.75) to node[midway, above,fill=white] {$G_4$} (6.75,4.25);
       \draw[bend right=30] (7.25,.25) to (8.75,.25);
       \draw[bend left=30] (7.25,4.25) to  (8.75,4.75);
       \draw (s11.center) to (s22.center) to (s32.center) to (s42.center) to  (8,1);
       \draw (s21.center) to (s31.center) to (s41.center) to (8,.5);
       \draw (s12.center) to (s23.center) to (s33.center) to (s43.center) to  (8,1.5);
       \draw (s13.center) to (s24.center) to (s34.center) to (s44.center) to  (8,2);
       \draw (s14.center) to (s25.center) to (s35.center) to (s45.center) to  (8,2.5);
       \draw (s26.center) to (s36.center) to (s46.center) to (8,3);
       \draw (s37.center) to (s47.center) to (8,3.5);
       \draw (s48.center) to (8,4);
       \draw[dashed] (8,.5) to (8.75,.5);
       \draw[dashed] (8,1) to (8.75,1);
       \draw[dashed] (8,1.5) to (8.75,1.5);
       \draw[dashed] (8,2) to (8.75,2);
       \draw[dashed] (8,2.5) to (8.75,2.5);
       \draw[dashed] (8,3) to (8.75,3);
       \draw[dashed] (8,3.5) to (8.75,3.5);
       \draw[dashed] (8,4) to (8.75,4);
       \draw[blue] (.75,.75) rectangle++(.5,2);
       \draw[blue] (2.75,.25) rectangle++(.5,3);
       \draw[blue] (4.75,.25) rectangle++(.5,3.5);
       \draw[blue] (6.75,.25) rectangle++(.5,4);
   \end{tikzpicture}
   \end{center}
\caption{Halin Depcomposition}
\label{fig:halindecomp}
\end{figure}

Also note that the union of all paths in $\bigcup\{\mathcal W_{i+1}:i\in\mathbb N\}$ is a graph that consists of a set $\cal W$ of disjoint rays. In general, these rays do not have to be in the same end of $G$. If $G$ has an end $\omega$ satisfying 
\begin{enumerate}[label=(H\arabic*),resume,leftmargin=15mm,itemsep=0pt,topsep=0pt]
\item\label{H6} \ all rays in $\cal W$ belong to $\omega$,
\end{enumerate} 
then the Halin decomposition is called an $\omega$-\emph{decomposition}. The following result characterizes all rays in an end $\omega$ when $G$ admits an $\omega$-decomposition.

\begin{theorem}\label{thm:displayray}
Let $\omega$ be an end of a graph $G$ and let $\{G_i : i \in\mathbb N\}$ be an $\omega$-decomposition of $G$. Then a ray $R$ of $G$ belongs to $\omega$ if and only if $R$ meets infinitely many $S_i$.
\end{theorem}
\begin{proof}
To prove the forward implication, suppose, for a contradiction, that some $R\in\omega$ intersects only finitely many $S_i$. Then, by \ref{H5}, $R$ is contained in $G_1\cup \ldots \cup G_n$ for some $n\in\mathbb N$. This, together with \ref{H5}, further implies that $S_n$ separates the infinite component of $R-S_n$ from the infinite component of $W-S_n$ for every $W\in\cal W$. Consequently, $R$ and $W$ are not in the same end, which contradicts \ref{H6}.

To prove the backward implication, suppose, for a contradiction, that some ray $R\not\in\omega$ intersects infinitely many $S_i$. Fix any $W_1\in\cal W$. Then, as $R$ and $W_1$ belong to different ends, $V(G)$ has a finite subset $Z$ and $G-Z$ has two components $G_r$ and $G_w$ such that $G_r$ contains the infinite component $R'$ of $R-Z$ and $G_w$ contains the infinite component $W_1'$ of $W_1-Z$. Since $Z$ is finite, by \ref{H1}, there exists $n\in\mathbb N$ such that $Z$ is contained in $G_1\cup \ldots \cup G_n$. On the other hand, since $R$ intersects infinitely many $S_i$, there exists an integer $m>n$ such that $R'$ intersects $S_m$. Take any $s\in V(R')\cap S_m$. Then, by \ref{H4}, $\cal W$ contains a unique ray $W_2$ that contains $s$. Let $W_2'$ be the subray of $W_2$ that starts with $s$. Since $m>n$, $W_2'$ is disjoint from $Z$. Meanwhile, since $W_2'$ and $R'$ share a common vertex $s$, $W_2'$ must be entirely contained in $G_r$. As a result, despite being from the same end, $W_1'$ and $W_2'$ are separated by $Z$, which is a contradiction. 
\end{proof}

Halin noted in \cite{halin2} that, in an $\omega$-decomposition, graphs $G_i$ don't have to be connected. Nevertheless, it is always possible to choose an $\omega$-decomposition with this connectivity property, as shown below. 

\begin{theorem}\label{thm:connectcomp}
Let $\omega$ be an end of a connected graph $G$ such that $G$ admits an $\omega$-decomposition. Then $G$ has an $\omega$-decomposition $\{G_i^*: i\in\mathbb N\}$ such that every $G_i^*$ is connected. 
\end{theorem}
\begin{proof}
We begin with an arbitrary $\omega$-decomposition $\{G_i: i\in\mathbb N\}$ of $G$. We first claim that for any $i\in\mathbb N$ and any $s_1,s_2\in S_i$, $F_i:=G_{i+1}\cup G_{i+2}\cup \ldots$ contains an $s_1s_2$-path. By \ref{H4}, $s_k$ ($k=1,2$) is contained in a unique ray $R_k$ of $\cal W$. Let $R_k'$ be the subray of $R_k$ that starts from $s_k$. Since $R_1'$ and $R_2'$ belong to the same end (by \ref{H6}), $G$ has infinitely many disjoint paths between them. Let $P$ be one of such paths that is disjoint from $S_i$ (path $P$ must exist since $S_i$ is finite). Then $P$ is contained in $F_i$ and thus $R_1'\cup R_2'\cup P$ contains an $s_1s_2$-path of $F_i$, which proves our claim.

For any $i,j\in\mathbb N$ with $i<j$, let $F_{i,j} = G_{i+1}\cup G_{i+2}\cup \ldots \cup G_j$. Then the above claim implies that, for each $i\in\mathbb N$, there exists $j\in\mathbb N$ with $j>i$ and such that, for any two vertices of $S_i$, $F_{i,j}$ contains a path between them. In other words, $S_i$ is entirely contained in a component of $F_{i,j}$. Let $\alpha(i)$ denote the smallest $j$ satisfying this property. Let $H_1=G_1$, $H_2=F_{1,\alpha(1)}$, and $H_i=F_{\alpha^{i-2}(1),\alpha^{i-1}(1)}$, for $i\ge3$. It is routine to verify that $\{H_i:i\in\mathbb N\}$ is an $\omega$-decomposition of $G$. Moreover, the choice of $\alpha$ also implies that each $V(H_i\cap H_{i+1})$ is entirely contained in a component of $H_{i+1}$. To simplify our notation, let us assume $H_i=G_i$ for all $i\in \mathbb N$. That is, each $S_i$ is entirely contained in a component $G_{i+1}'$ of $G_{i+1}$. Note that all paths in $\mathcal W_{i+1}$ are contained in $G_{i+1}'$. 

For each $i\in \mathbb N$, let $G_{i+1}'':=G_{i+1}-G_{i+1}'$. Note that $G_{i+1}''$ could be empty and $G_{i+1}$ is the disjoint union of $G_{i+1}'$ and $G_{i+1}''$. Let $G_1''=G_1$. Then, since $G$ is connected, each component of $G_i''$ must intersect $S_i$. Therefore, $G_i^*:=G_i''\cup G_{i+1}'$ is connected for every $i\in\mathbb N$. Now it is routine to verify that $\{G_i^*: i\in\mathbb N\}$ is an $\omega$-decomposition of $G$, which proves the theorem.
\end{proof}

In addition to introducing $\omega$-decomposition, Halin also established \cite{halin2} its existence for locally finite graphs.  

\begin{theorem}\label{thm:halin}
Every locally finite connected graph $G$ admits an $\omega$-decomposition for every end $\omega$ of $G$. 
\end{theorem}

We may now state and prove the main result of this section. 

\begin{theorem}\label{thm:locfin}
Let $\omega$ be an end of an infinite locally-finite 2-connected graph $G$. Then $G$ contains as an induced subgraph an infinite slim ladder or a member of $\Loneinf \cup\Ltwoinf \cup \Lthreeinf$ such that its rails are in $\omega$.
\end{theorem}
\begin{proof}
For any path $P=v_1v_2\ldots v_n$ of $G$, we denote by $\mathring{P}=P(v_1,v_n)$. If $X,Y\subseteq V(G)$ are disjoint, we call $P$ an $XY$-{\it path} if $V(P)\cap X=\{v_1\}$ and $V(P)\cap Y=\{v_n\}$. We emphasize that, in this case, $V(\mathring{P})\cap(X\cup Y) = \emptyset$. By \cref{thm:halin,thm:connectcomp}, $G$ has an $\omega$-decomposition $\{G_i: i\in\mathbb N\}$ such that each $G_i$ is connected. Let $F_0=G_1$ and $F_i=G_{3i-1}\cup G_{3i}\cup G_{3i+1}$ for all $i\in\mathbb N$. Note that $V(F_i\cap F_{i-1})=S_{3i-2}$ and $V(F_i\cap F_{i+1})=S_{3i+1}$. For any disjoint $S_{3i-2}S_{3i+1}$-paths $P,Q$ of $F_i$, it is clear that both $\mathring{P}$ and $\mathring{Q}$ are contained in $F_i-(S_{3i-2}\cup S_{3i+1})$. In addition, since $G_{3i}$ is connected and, by \ref{H5}, both $\mathring{P}$ and $\mathring{Q}$ meet $G_{3i}$, we conclude that 
 
\begin{claim}\label{cl:1}$\mathring{P}$ and $\mathring{Q}$ are contained in the same component of $F_i-(S_{3i-2}\cup S_{3i+1})$.\end{claim}

Since $G$ is 2-connected, \ref{H5} implies that $|S_1|\ge 2$. This allows us to fix two distinct vertices $x_1$ and $y_1$ in $S_1$. Suppose that distinct $x_i,y_i\in S_{3i-2}$ have been selected. We choose $x_{i+1},y_{i+1}\in S_{3i+1}$ as follows. By \ref{H4}, $F_i-(S_{3i-2}-\{x_i,y_i\})$ contains two disjoint paths from $\{x_i,y_i\}$ to $S_{3i+1}$. Thus, we can choose in $F_i$ two disjoint $S_{3i-2}S_{3i+1}$-paths $P_i$ and $Q_i$  satisfying the following: 
\begin{enumerate}[label=(\roman*),leftmargin=15mm,itemsep=0pt,topsep=0pt] 
\item\label{i}  $V(P_i)\cap S_{3i-2}=\{x_i\}$ and $V(Q_i)\cap S_{3i-2}=\{y_i\}$, 
\item \label{ii} $|E(P_i\cup Q_i)|$ is minimal, and  
\item\label{iii} subject to (i-ii), the distance between $\mathring{P}_i$ and $\mathring{Q}_i$ in $F_i-(S_{3i-2}\cup S_{3i+1})$ is minimal (which can be achieved because of \cref{cl:1}).
\end{enumerate}
Let $x_{i+1}$ and $y_{i+1}$ be the ends of $P_i$ and $Q_i$ in $S_{3i+1}$, respectively. This means we can define inductively $P_i$ and $Q_i$ for all $i\in\mathbb N$. Let $P:=\cup\{P_i:i\in\mathbb N\}$ and $Q:=\cup\{Q_i:i\in\mathbb N\}$. Then $P$ and $Q$ are disjoint rays of $G$. By \cref{thm:displayray}, both $P$ and $Q$ belong to $\omega$. In addition, both $P$ and $Q$ are induced rays. To see this, suppose, say, that $P$ has a chord $e=uv$. By \ref{H5}, some $F_i$ contains both $u$ and $v$. On the other hand, by \ref{ii}, at least one of $u$ and $v$, say, $u$, is not on $P_i$. This means that either $P_{i-1}$ or $P_{i+1}$, say, $P_{i+1}$, contains $u$. As a result, $V(P_{i+1})\cap S_{3i+1}=V(P_{i+1}\cap F_i)\supseteq\{u,x_{i+1}\}$, which contradicts \ref{i}, as $u\ne x_{i+1}$ (and this is because $u\not\in V(P_i)$ but $x_{i+1}\in V(P_i)$). Therefore, $P$ and $Q$ are disjoint induced rays of $\omega$.

Let $P=p_1p_2\cdots$ and $Q=q_1q_2\cdots$. We first consider the case that $G$ has infinitely many edges between $P$ and $Q$. For each edge $e=p_mq_n$ of $G$, we prove that $G$ does not have an edge $e'=p_{m'}q_{n'}$ satisfying $(m-m')(n-n')<-1$. Suppose such $e'$ exists. By \ref{H5}, there exists $i$ such that $P_i\cup Q_i$ contains both $p_m$ and $q_n$. Then \ref{H5} and $(m-m')(n-n')<-1$ imply that $P_i\cup Q_i$ also contains $p_{m'}$ and $q_{n'}$. However, replacing $P_i[p_m,p_{m'}]\cup Q_i[q_n,q_{n'}]$ in $P_i\cup Q_i$ by $e$ and $e'$ would create two disjoint $S_{3i-2}S_{3i+1}$-paths of shorter total length, which contradicts \ref{ii}. Therefore, $(m-m')(n-n')\ge-1$ holds for all $e,e'$.
Let $G'$ be the subgraph of $G$ induced by $V(P\cup Q)$.
If $p_1q_1\in E(G)$, then $G'$ is a slim ladder, and the conclusion follows.
Otherwise, $G'+p_1q_1$ is a slim ladder, and thus, $G'$ contains an induced slim ladder. 
So $G$ contains an induced slim ladder, and the conclusion follows.

Next, without loss of generality, we assume that $G$ has no edges between $P$ and $Q$. For each $i\in\mathbb N$, let $R_i=r_1r_2\ldots, r_k$ be one of the shortest paths between $\mathring{P}_i$ and $\mathring{Q}_i$ in $F_i-(S_{3i-2}\cup S_{3i+1})$, where $r_1\in V(P_i)$ and $r_k\in V(Q_i)$. By \ref{H5}, no vertex of $R_i$ is adjacent to any vertex of $P\cup Q$ outside $P_i\cup Q_i$. In addition, since $R_i$ is the shortest, it is an induced path, and $r_2$ is its only vertex with neighbors on $P_i$ and $r_{k-1}$ is its only vertex with neighbors on $Q_i$. Note that $r_2=r_{k-1}$ when $k=3$. Let $p_m$ and $p_{m'}$ ($m\le m'$) be the (possibly identical) two neighbors of $r_2$ such that all neighbors of $r_2$ on $P_i$ are contained in $P_i[p_m,p_{m'}]$. If $m'-m\ge2$, then \ref{ii} or \ref{iii} is violated since subpath $P_i[p_m,p_{m'}]$ of $P_i$ could be replaced by path $p_mr_2p_{m'}$. Therefore, $r_2$ has either only one neighbor or only two adjacent neighbors on $P_i$. We may say the same about neighbors of $r_{k-1}$ on $Q_i$. As a result, there are four types of $R_i$ depending on their number of neighbors on $P\cup Q$. By the pigeon-hole principle, one of these four types happens infinitely often.
If there are infinitely many $i$ such that $r^i_2$ has only one neighbor on $P$ and $r^i_{k-1}$ has only one neighbor on $Q$, then $P\cup Q\cup R_1\cup R_2\cup \cdots$ contains a member of $\Loneinf$, and the conclusion follows.
If there are infinitely many $i$ such that $r^i_2$ has two neighbors on $P$ and $r^i_{k-1}$ has only one neighbor on $Q$, then $P\cup Q\cup R_1\cup R_2\cup \cdots$ contains a member of $\Ltwoinf$, and the conclusion follows.
If there are infinitely many $i$ such that $r^i_2$ has only one neighbor on $P$ and $r^i_{k-1}$ has two neighbors on $Q$, then $P\cup Q\cup R_1\cup R_2\cup \cdots$ contains a member of $\Ltwoinf$, and the conclusion follows.
If there are infinitely many $i$ such that $r^i_2$ has two neighbors on $P$ and $r^i_{k-1}$ has two neighbors on $Q$, then $P\cup Q\cup R_1\cup R_2\cup \cdots$ contains a member of $\Lthreeinf$, and the conclusion follows.
\end{proof}

\section{Proving \cref{thm:infinite,thm:finite}}\label{sec:provingmainthm}
We can now prove \cref{thm:infinite} restated below.
 
\thminftwocon*
\begin{proof}
	Let $G$ be an infinite 2-connected graph.
	If $G$ has a vertex of infinite degree, then \cref{lem:infvert} implies that $G$ contains as an induced subgraph either $K_{\infty}$ or a member of $\TFinf\cup\TFinf^{\Delta}\cup\Theta_{\infty}$, and the conclusion follows.
	
	We may therefore assume that $G$ is locally-finite.
        \cref{thm:locfin} implies that $G$ contains as an induced subgraph either an infinite slim ladder or a member of $\Loneinf\cup\Ltwoinf\cup\Lthreeinf$, and the conclusion follows.
\end{proof}

  We now present an alternative proof of \cref{thm:finite} to that in
    \cite{unavoidableinducedsubgraphs} using \cref{thm:infinite}.
  To do this, we need the following:

  \begin{lemma}\label{lem:nobigcycle}
    Let $c$ be an integer exceeding two, and let~$G$ be a $2$-connected finite graph
      with an $n$-element vertex set~$V$, and with no induced cycle of length exceeding $c$.
    Then $V$ may be enumerated as $v_1$, $v_2$, $\dots$,~$v_n$ so that for each integer
      $i$ satisfying $c\le i\le n$, there is a $j$ such that $i-(c-3)\le j\le i$ and
      $G\left[\{v_1,v_2,\dots,v_j\}\right]$ is $2$-connected.
  \end{lemma}

  \begin{proof}
    Let $c$, $n$, $G$, and $V$ be as described in the theorem.
    We begin by using induction to construct a sequence of induced cycles
      $D_0$, $D_1$, \dots,~$D_t$ of $G$, as follows.
    Let $D_0$ be an arbitrary induced cycle of $G$.
    Such a cycle exists because $G$ is $2$-connected.
    Now assume that the cycles $D_0$, $D_1$, \dots,~$D_i$ have been defined so that
      $G_i=D_0\cup D_1\cup\cdots\cup D_i$ is a $2$-connected subgraph of $G$.
      Let $V_i=V[G_i]$.
    Since $G$ is $2$-connected, for every vertex $v$ not in $G_{i}$ there is a path in $G$
      having $v$ as an internal vertex, and meeting $G_{i}$ only at its endpoints.
    Moreover, since $G_{i}$ is connected, every such path may be completed to a cycle
      by appending a path contained in~$G[V_i]$.
    From among all possible cycles constructed in this way, over all vertices $v$ outside $G_i$, choose one of minimum length and name
      it~$D_{i+1}$.
    It follows that $D_{i+1}$ is induced in $G$, and so its length is at most~$c$.
    Now define $G_{i+1} = G_{i} \cup D_{i+1}$ and observe that $G_{i+1}$ is $2$-connected
      and adds to $G_{i}$ at most $c-2$ new vertices.
    If $V(G_{i+1}) = V(G)$, then the inductive process terminates;
      otherwise it continues.

    Now, to enumerate the vertices of $G$, proceed in steps, with the first step being listing
      the vertices of $D_0$, in arbitrary order, and with step $i$, for $i\ge 1$, consisting of
      appending to the list the vertices of $D_i$ not yet on the list.
    Note that each step adds at most $c-2$ vertices to the list, and at the completion of every
      step, the list produced thus far induces in $G$ a $2$-connected subgraph.
    The conclusion follows.
  \end{proof}

We say that a graph is \emph{$\{K_r, \Theta_r, \Lambda_r\}$-free} if it contains no member of $\{K_r\}\cup \Theta_r\cup \Lambda_r$ as an induced subgraph.
We can now present a proof of \cref{thm:finite} using \cref{thm:infinite}.
\thmfinite*

\begin{proof}
    We will prove \cref{thm:finite} by contradiction.
    Suppose that $f_{\ref{thm:finite}}(r)$ does not exist for some $r$.
    Then for every $n\ge 3$, there exists a 2-connected graph, $G_n$, of order at least $n$ that is $\{K_r,\Theta_r,\Lambda_r\}$-free.

    Note that an induced cycle of order at least $r+2$ is a member of the family $\Lambda_r$.
    Thus, we may apply \cref{lem:nobigcycle} with $c=r+1$ to each $G_n$.
    Consider an infinite sequence $v_1,v_2,\dots $.  We use this sequence to label the  vertices of each of the $G_n$. 
    Enumerate the vertices of each $G_n$ by $v_1$, $v_2$, \dots, $v_{|V(G_n)|}$ such that for each $i$ satisfying $r< i \le |V(G_n)|$, there exists a $j$ such that $i-(r-2)\le j\le i$ and $G_n[v_1, v_2, \dots, v_j]$ is 2-connected.

    We will inductively construct an infinite graph $G$ on the vertex set $u_1$, $u_2$, \dots\space such that for every positive integer $m$, the graph $G[u_1,u_2,\dots,u_m]$ is isomorphic, using the mapping $\pi:u_j\rightarrow v_j$ for all $j\in \N$, to infinitely many graphs of the form $G_i[v_1, v_2,\dots, v_m]$ for $i\ge m$. Then we show that $G$ is 2-connected in order to apply \cref{thm:infinite}.
    
    Let $\ell=1$, then let $I_{1}=\mathbb{N}\setminus \{1,2\}$ and let $G[u_1]\cong G_n[v_1]$ for $n\in I_1$.
    Inductively, suppose that $I_{\ell}$ has been defined (and is infinite) and that $G[u_1,u_2,\dots,u_{\ell}]$ has been defined so that it is isomorphic under $\pi$ to $G_i[v_1,v_2,\dots, v_{\ell}]$ for all $i\in I_{\ell}$.
    Since $I_{\ell}$ is infinite and $|G_i|>\ell$ for all $i>\ell$, it follows that $I_{\ell}$ contains infinitely many $i$ such that $|G_i|>\ell$.
    So, there is an infinite $I_{\ell+1}\subseteq I_{\ell}$ such that $G_i[v_1,v_2,\dots, v_{\ell+1}]= G_j[v_1,v_2,\dots, v_{\ell+1}]$ for all $i$ and $j$ in $I_{\ell+1}$.
    Let $G[u_1, u_2, \dots,u_{\ell+1}]$ be congruent to $ G_i[v_1,v_2,\dots, v_{\ell+1}]$ under $\pi$ for some $i\in I_{\ell+1}$.
    Note that it does not matter which $i\in I_{\ell+1}$ by selection of $I_{\ell+1}$.
    Further, $G[u_1, u_2,\dots, u_{\ell+1}]$ is $\{K_r,\Theta_r,\Lambda_r\}$-free because $G_i[v_1,v_2,\dots, v_{\ell+1}]$ is $\{K_r,\Theta_r,\Lambda_r\}$-free.

    Note that $G$ is $\{K_r,\Theta_r,\Lambda_r\}$-free by construction.
    To apply \cref{thm:infinite}, it remains to show that $G$ is 2-connected.
    
    At step $t+r$ for a natural number $t$, we have that $I_{t+r}\subseteq \{j~:~G_j[v_1,v_2,\dots, v_{t+r}]\cong G[u_1,u_2,\dots, u_{t+r}]\}$.
    So there exists a graph $G_s$ such that $G_s[v_1,v_2,\dots, v_{t+r}]$  is isomorphic to $G[u_1,u_2,\dots, u_{t+r}]$.
    By construction, $G_s[v_1,v_2,\dots ,v_{t+r'}]$ for some $2\le r' \le r$ is 2-connected.
    Since we can do this for any integer $t$ exceeding 2, it follows that $G$ is 2-connected.
    
    Since $G$ is an infinite 2-connected graph, \cref{thm:infinite} implies that $G$ contains one of the following as an induced subgraph: $K_{\infty}$, an infinite slim ladder, and a member of one of the following families: $\Theta_{\infty}$, $\TFinf$, $\TFinf^{\Delta}$, $\Loneinf$, $\Ltwoinf$, and $\Lthreeinf$.
    
    If $G$ contains an induced $K_{\infty}$, then $G$ contains an induced $K_r$; a contradiction.
    If $G$ contains an infinite induced slim ladder, then $G$ contains an induced member of the family $\Lambda_r$; a contradiction.
     If $G$ contains an induced member of the family $\Theta_{\infty}$, then $G$ contains an induced member of the family $\Theta_r$; a contradiction.
    If $G$ contains an induced member of $\TFinf^\Delta\cup\Loneinf \cup \Ltwoinf\cup\Lthreeinf$, then $G$ contains an induced cycle of length at least $r+2$.  Thus $G$ contains an induced member of $\Lambda_r$; a contradiction.
    If $G$ contains an induced member $G'$ of $\TFinf$, then let $C$ be a cycle of $G'$ of length at least $r+2$.  Then the subgraph of $G$ of induced by $V(C)$ is a member of the family $\Lambda_r$.

    Hence, for every positive integer $r$ exceeding two, there is an integer $f_{\ref{thm:finite}}(r)$ such that every 2-connected graph of order at least $f_{\ref{thm:finite}}(r)$ contains $K_r$ or a member of $\Theta_r\cup \Lambda_r$ as an induced subgraph, as required.
\end{proof}

\section*{Acknowledgements}
	The authors express gratitude to the referee for carefully reading the paper and for suggesting an alternative proof of \cref{thm:infinite}, which significantly simplifies our original proof.

\section*{Declaration of Competing Interests}
The authors declare that they have no known competing financial interests or personal relationships that could have appeared to influence the work reported in this paper.

\bibliographystyle{plain}
\bibliography{UnavoidableBib}

\begin{thebibliography}{1}

\bibitem{unavoidableinducedsubgraphs}
S.~Allred, G.~Ding, and B.~Oporowski.
\newblock Unavoidable induced subgraphs of large 2-connected graphs.
\newblock {\em SIAM Journal on Discrete Mathematics}, 37(2):684--698, 2023.

\bibitem{homogeneous}
M.~Chudnovsky, R.~Kim, S.~Oum, and P.~Seymour.
\newblock Unavoidable induced subgraphs in large graphs with no homogeneous
  sets.
\newblock {\em J. Combin. Theory Ser. B}, 118:1--12, 2016.

\bibitem{unavoidableparminor4conngraphs}
C.~Chun, G.~Ding, B.~Oporowski, and D.~Vertigan.
\newblock Unavoidable parallel minors of 4-connected graphs.
\newblock {\em J. Graph Theory}, 60(4):313--326, 2009.

\bibitem{doublecon}
G.~Ding and P.~Chen.
\newblock Unavoidable doubly connected large graphs.
\newblock {\em Discrete Math.}, 280(1):1--12, 2004.

\bibitem{Halin2}
R.~Halin.
\newblock {\em Systeme Disjunkter Unendlicher Wege in Graphen}, pages 55--67.
\newblock Birkh{\"a}user Basel, Basel, 1977.

\bibitem{konig}
D.~K{\"o}nig.
\newblock {\"U}ber eine schlussweise aus dem endlichen ins unendliche.
\newblock {\em Acta Sci. Math. (Szeged)}, 3(2-3):121--130, 1927.

\bibitem{Unavoidabletopminor3conngraphs}
B.~Oporowski, J.~Oxley, and R.~Thomas.
\newblock Typical subgraphs of {$3$}- and {$4$}-connected graphs.
\newblock {\em J. Combin. Theory Ser. B}, 57(2):239--257, 1993.

\bibitem{ramsey}
F.~Ramsey.
\newblock On a {p}roblem of {f}ormal {l}ogic.
\newblock {\em Proc. London Math. Soc.}, 30:264--286, 1930.

\bibitem{west}
D.~West.
\newblock {\em Introduction to Graph Theory}.
\newblock Prentice Hall, Upper Saddle River, N.J., second edition, 2001.

\end{thebibliography}
\vspace{-5mm}

\end{document}